\documentclass[12pt]{amsart}

\usepackage{graphicx}
\usepackage{color}

\usepackage{amssymb}
\usepackage{amsthm}
\usepackage{url}
\usepackage{hyperref}
  \hypersetup{backref,colorlinks=true}
\usepackage{a4wide}
\usepackage[small,nohug,heads=vee]{diagrams}
\diagramstyle[labelstyle=\scriptstyle]

  \makeatletter
  \CheckCommand*\refstepcounter[1]{\stepcounter{#1}%
      \protected@edef\@currentlabel
       {\csname p@#1\endcsname\csname the#1\endcsname}%
  }
  \renewcommand*\refstepcounter[1]{\stepcounter{#1}%
    \protected@edef\@currentlabel
      {\csname p@#1\expandafter\endcsname\csname the#1\endcsname}%
  }
  \def\labelformat#1{\expandafter\def\csname p@#1\endcsname##1}
  \DeclareRobustCommand\Ref[1]{\protected@edef\@tempa{\ref{#1}}%
     \expandafter\MakeUppercase\@tempa
  }
  \makeatother

  \makeatletter
  \newcommand{\numberlike}[2]{%
     \expandafter\def\csname c@#1\endcsname{%
         \expandafter\csname c@#2\endcsname}%
  }
  \makeatother

  \def\DefaultNumberTheoremWithin{section}

\def\aA{{\mathcal A}}

\def\cC{{\mathcal C}}
\def\zZ{{\mathcal Z}}
\def\mM{{\mathcal M}}
\def\RR{{\mathbb R}}
\def\KK{{\mathbb K}}
\def\CC{{\mathbb C}}

\def\Crit{{\mathrm{Crit}}}
\def\link{{\mathrm{link}}}
\def\sgn{{\mathrm{sgn}}}
\def\supp{{\mathrm{supp}}}
\def\star{{\mathrm{star}}}
\def\hocolim{{\mathrm{hocolim}}}
\def\colim{{\mathrm{colim}}}
\def\cells{{\bullet}}
\def\ncell{{\bullet}}
\def\dual{{\circ}}
\def\join{{\ast}}
\def\Tor{{\mathrm{Tor}}}

\def\Homology{{\mathrm{H}}}
\def\pt{{\mathrm{pt}}}
\def\id{{\mathrm{id}}}

\theoremstyle{plain}
  \newtheorem{Lemma}{Lemma}
     \numberwithin{Lemma}{\DefaultNumberTheoremWithin}
     \labelformat{Lemma}{Lemma~#1}
  
     \numberwithin{Claim}{\DefaultNumberTheoremWithin}
     \numberlike{Claim}{Lemma}
     \labelformat{Claim}{Claim~#1}
  \newtheorem{Theorem}{Theorem}
     \numberwithin{Theorem}{\DefaultNumberTheoremWithin}
     \numberlike{Theorem}{Lemma}
     \labelformat{Theorem}{Theorem~#1}
  \newtheorem{Corollary}{Corollary}
     \numberwithin{Corollary}{\DefaultNumberTheoremWithin}
     \numberlike{Corollary}{Lemma}
     \labelformat{Corollary}{Corollary~#1}
  \newtheorem{Proposition}{Proposition}
     \numberwithin{Proposition}{\DefaultNumberTheoremWithin}
     \numberlike{Proposition}{Lemma}
     \labelformat{Proposition}{Proposition~#1}
  
     \numberwithin{Conjecture}{\DefaultNumberTheoremWithin}
     \numberlike{Conjecture}{Lemma}
     \labelformat{Conjecture}{Conjecture~#1}

  \theoremstyle{definition}
  
     \numberwithin{Definition}{\DefaultNumberTheoremWithin}
     \numberlike{Definition}{Lemma}
     \labelformat{Definition}{Definition~#1}

  \theoremstyle{definition}
  
     \numberwithin{Question}{\DefaultNumberTheoremWithin}
     \numberlike{Question}{Lemma}
     \labelformat{Question}{Question~#1}

  \theoremstyle{definition}
  
     \numberwithin{Problem}{\DefaultNumberTheoremWithin}
     \numberlike{Problem}{Lemma}
     \labelformat{Problem}{Problem~#1}

  \theoremstyle{remark}
  \newtheorem{Remark}{Remark}
     \numberwithin{Remark}{\DefaultNumberTheoremWithin}
     \numberlike{Remark}{Lemma}
     \labelformat{Remark}{Remark~#1}
  \theoremstyle{remark}

  \newtheorem{Example}{Example}
     \numberwithin{Example}{\DefaultNumberTheoremWithin}
     \numberlike{Example}{Lemma}
     \labelformat{Example}{Example~#1}
  
     \labelformat{Case}{Case~#1}
     \numberwithin{Case}{Lemma}
  
     \labelformat{Step}{Step~#1}
     \numberwithin{Step}{Lemma}

\begin{document}

\title{Discrete Morse theory for moment-angle complexes of pairs $(D^n,S^{n-1})$}

\author{Vladimir Gruji\'c}
  \address{Faculty of Mathematics, Belgrade University, Studentski trg 16, 11000 Belgrade, Serbia}
  \email{vgrujic@matf.bg.ac.rs}

\author{Volkmar Welker}
  \address{Philipps-Universit\"at Marburg, Fachbereich Mathematik und Informatik, Hans-Meerwein-Strasse, 35032 Marburg, Germany}
  \email{welker@mathematik.uni-marburg.de}

\thanks{First author was supported by a DAAD postdoctoral fellowship and by Ministry of Science of Republic of Serbia,
  project 174034}

\begin{abstract}
  For a finite simplicial complex $K$ and a CW-pair $(X,A)$, there
  is an associated CW-complex $\zZ_K(X,A),$ known as a
  polyhedral product. We apply discrete Morse theory to a
  particular CW-structure on the $n$-sphere
  moment-angle complexes $\zZ_K(D^{n}, S^{n-1})$.
  For the class of simplicial complexes with vertex-decomposable
  duals, we show that the associated $n$-sphere moment-angle complexes
  have the homotopy type of wedges of spheres. As a corollary we show that
  a sufficiently high suspension of any restriction of a simplicial complex with
  vertex-decomposable dual is homotopy equivalent to a wedge of spheres.
\end{abstract}

\keywords{simplicial complex, $n$-sphere moment-angle complex, coordinate subspace arrangement, discrete Morse theory}

\subjclass{55P15 52C35 05E45}


\maketitle
\section{Introduction}

  In this paper we study the following construction of topological spaces.
  Let $K$ be a simplicial complex $K$ over ground set $[m]:= \{1,\ldots, m\}$ and let $(X,A)$ be a
  pair of spaces. For a simplex $\sigma\in K$ set
  \[(X,A)^{\sigma}:=\{(x_1,\ldots,x_m)\in X^{m}\mid x_i\in A,
  i\notin\sigma\}.\] and define $\zZ_K(X,A) \subseteq X^m$ as
  \[\zZ_K(X,A)=\bigcup_{\sigma\in K}(X,A)^{\sigma}\]
  equipped with the subspace topology of the Cartesian product
  $X^{m}$. The space $\zZ_K(X,A)$ is called the
  {\it polyhedral product} of $K$ with respect to the pair $(X,A)$.

  Polyhedral products are a natural generalization of Cartesian products
  and appear prominently in various areas of mathematics
  (see \cite{BP}, \cite{GT1}, \cite{GT2}, \cite{GPTW}, \cite{IK} and \cite{BBCG} for references and details).
  Despite its simple combinatorial definition the topology of $\zZ_K(X,A)$ is
  hard to control. The most far reaching general results on the homotopy type can be found in
  \cite{GT2} where a wedge decomposition in terms of the homotopy type of induced
  subcomplexes of $K$ is given for shifted simplicial complexes $K$ and
  pairs $(CX,X)$ of a cone $CX$ over $X$, and in \cite{GPTW} where among others it is proved that
  for simplicial complexes $K$ whose non-faces form a chordal graph $\zZ_K(D^2,S^{1})$ has the homotopy
  type of a wedge of spheres.
  Of particular interest is the case when $(X,A) = (D^n,S^{n-1})$ for some
  number $n \geq 1$, where $D^n$ is the $n$-dimensional disk and $S^{n-1}$
  its bounding $(n-1)$-sphere.
  Among those the case $n=2$ plays the most important role since it relates to torus action (see \cite{BP})
  and coordinate subspace arrangements (see Section \ref{sec:basics}).
  Extending standard terminology from the case $n = 2$ we call the space $\zZ_K(D^{n}, S^{n-1})$,
  the {\it $n$-sphere moment-angle complex}.

  We extend the result from \cite{GT2} (see also \cite{IK}) in the case $(X,A) = (D^n,S^{n-1})$
  from the class of shifted simplicial complexes to the class of simplicial complexes $K$ for
  which its {\it Alexander dual}  $K^\dual :=
  \{\sigma\subset[m]\mid[m]\setminus\sigma\notin K\}$
  is vertex decomposable.
  We also refer the reader to Section \ref{sec:vertex} for unexplained
  terminology.  Our main result states:

  \begin{Theorem}\label{t1}
    Let $K$ be a simplicial complex on ground set $[m]$
    and $\{i\} in K$ for all $i \in [m]$
    with vertex decomposable Alexander dual complex $K^\dual$.
    Then the $n$-sphere moment angle complex
    $\zZ_K(D^{n}, S^{n-1})$ has a homotopy type of the wedge
    of spheres

    \begin{eqnarray} \label{eq:wedge}
     \zZ_K(D^{n}, S^{n-1}) & \simeq & \bigvee_{i \geq 0}
        \bigvee_{{{M \not\in K} \atop {M \subseteq [m]}}}
          \dim_\KK \widetilde{\Homology}_{i-(n-1)\#M -1} (K_M ,\KK) \cdot S^i
    \end{eqnarray}

     for any field $\KK$, where $K_M$ is the restriction of $K$ to $M$.
  \end{Theorem}

  The intuition for the theorem comes from a result by Buchshta\-ber and Panov
  \cite[Thm. 7.7]{BP}. There they prove that the cohomology
  ring of $\zZ_K(D^{2}, S^{1})$ is isomorphic to the
  $\Tor$-algebra of the Stanley-Reisner ring of the simplicial complex $K$ extending
  an additive isomorphism from \cite{GPW}. From \cite{HRW} we know that
  if $K^\dual$ is a (non-pure) shellable simplicial complex (see \cite{BW1})
  and the Stanley-Reisner
  ideal of $K$ is generated in degrees $\geq 2$ (i.e. $\{i\} \in K$ for all $i \in [m]$)
  then
  the Stanley-Reisner ring of $K$ is Golod over all fields.
  The Golod property \cite{GL} says that all Massey operation on the Koszul complex
  vanish -- for which the first is the product on the $\Tor$-algebra.
  By the work of Berglund and J\"ollenbeck \cite{BJ} for
  Stanley-Reisner rings the Golod property actually is equivalent to the vanishing
  of the product on the $\Tor$-algebra.
  Since vertex-decomposable simplicial complexes are shellable \cite[Thm. 11.3]{BW} and
  hence sequentially Cohen-Macaulay it follows
  from \cite[Thm. 4]{HRW} that
  the product on the cohomology ring of $\zZ_K(D^{2}, S^{1})$ must be trivial
  if $K^\dual$ is vertex-decomposable. Indeed in \cite[Thm. 6]{B} it is proved
  that $\zZ_K(D^{2}, S^{1})$ is rationally homotopy equivalent to a wedge
  of spheres if and only if $K$ the rational Stanley-Reisner ring of $K$
  is Golod. But this provides insight only for the
  case $n =2$. The fact that $\zZ_K(D^{n}, S^{n-1})$ is homotopy equivalent
  to a wedge of spheres for all $n \geq 1$ for $K$ with vertex-decomposable dual
  remains as mysterious as the previous
  results from \cite{GT2} for the subclass of shifted simplicial complexes and pairs
  $(CX,X)$, where $CX$ is the cone over $X$.
  Nevertheless, the homotopy equivalence from \ref{t1} is known to hold in the stable category
  using a single suspension only.
  It follows from \cite[Cor. 2.24]{BBCG} (see \ref{pr:susp}) below) that
  for any $K$ and any $(X,A)$ the suspension of $\zZ_K(X,A)$ has  a homotopy type
  given by the suspension of formula from \ref{t1}.

  A simple calculation shows that the Alexander dual of a shifted simplicial complex is also shifted.
  Thus by \cite[Thm. 11.3]{BW} shifted simplicial complexes have vertex decomposable Alexander dual complexes.
  Hence the result of \ref{t1} contains the main result of \cite{GT2} for the important case
  $(CX,X) = (D^n,S^{n-1})$. For the subclass of skeleta $K$ of simplices it was
  previously shown in \cite{GT1} that the moment-angle complexes
  $\zZ_K(D^{2},S^{1})$ have the homotopy type of wedges of
  spheres.

  Finally, using a result by Eagon and Reiner \cite[Prop. 8]{ER}
  we obtain as a corollary the following slight extension of a result from \cite{GPTW}.
  which is the main new result of their Theorem 4.6.
  For its formulation we denote by $K^{(1)}$ the {\it $1$-skeleton} of the simplicial complex $K$
  which in turn can be considered a graph.
  Now $K^{(1)}$ is called {\it chordal} if all cycles of length
  $\geq 4$ have a chord.
  Recall that a simplicial complex $K$ if called {\it flag} is all its minimal
  non-faces are of size $2$.

  \begin{Corollary}[see Theorem 4.6 \cite{GPTW}] \label{co:flag}
    Let $K$ be a flag simplicial complex such that $K^{(1)}$ is a chordal graph.
    Then $\zZ_K(D^n,S^{n-1})$ is homotopy equivalent to a wedge of spheres.
  \end{Corollary}

  We note that \cite[Prop. 7]{ER} implies that if the minimal non-faces of $K$
  are supported on the bases of a matroid then $K^\dual$ is vertex-decomposable.
  Similarly, the work of Billera and Provan \cite{PB} shows that
  quite a large class of those simplicial complexes for which $K^\dual$ is the boundary complex of a
  simplicial polytope satisfies the assumptions of \ref{t1}. Thus in either case we
  obtain that $\zZ_K(D^n,S^{n-1})$ is homotopy equivalent to a wedge of spheres.

  In contrast to \cite{GT1} and \cite{GT2}, where more advanced tools from
  homotopy theory are invoked, our methods are combinatorial and elementary.
  We apply discrete Morse theory \cite{F} to a suitably chosen regular cell
  decomposition and some basic lemmas from homotopy theory. Indeed, we show in
  \ref{co:appl} (iii) that the critical cells arising from our application of
  discrete Morse theory all contribute a sphere. Since by \ref{t2} the space
  $\zZ_K(D^{n}, S^{n-1})$ is homotopy equivalent to a
  complement of a subspace arrangement this shows that our application of discrete Morse
  theory yields results that are similar in spirit to the ones in \cite{SS} and \cite{D}.
  There a cell decomposition for complements of arbitrary complex hyperplane arrangement
  is constructed such that the cells from a basis of homology.

  Finally, as an interesting consequence in geometric combinatorics we obtain the
  following corollary of \ref{t1}.

  \begin{Corollary}
     \label{co:susp}
     Let $K$ be a simplicial complex on ground set $[m]$ with vertex decomposable Alexander dual complex $K^\dual$.
     Then for any $M \subseteq [m]$ the $(\# M +2)$-fold suspension $\Sigma^{\#M +2} K_M$ of
     the restriction $K_M$ of $K$ to $M$ is homotopy equivalent to a wedge of
     spheres.
  \end{Corollary}

  The paper is organized as follows.
  In Section \ref{sec:basics}, we provide basic facts about
  vertex decomposable simplicial complexes and polyhedral products.
  In particular, we describe the relation of polyhedral products and
  coordinate subspace arrangements.
  In Section \ref{sec:dimo}, we first review Forman's discrete
  Morse theory \cite{F} in terms of acyclic matching. Then for an
  arbitrary simplicial complex $K$, we construct the matching
  $\mM_K$ corresponding to an $n$-sphere polyhedral product
  complex $\zZ_K(D^{n}, S^{n-1})$ and identify the set of
  critical cells $\Crit(\mM_K)$. In Section \ref{sec:proof}, we use the
  results from Section \ref{sec:dimo} to complete the proofs of
  \ref{t1}, \ref{co:flag} and \ref{co:susp}.

\section{Basics about vertex decomposable complexes and polyhedral products}
  \label{sec:basics}

  \subsection{Vertex Decomposable Simplicial Complexes}
  \label{sec:vertex}

  First we recall some basic terminology for simplicial complexes.
  Let $K \subseteq 2^{[m]}$ be a simplicial complex over the
  ground set $[m]$. The elements of $K$ are called {\it faces} and the inclusionwise
  faces are called {\it facets} of $K$. For $M \subseteq [m]$ we write $K_M :=
  \{ \sigma \in K~|~\sigma \subseteq M\}$ for the
  {\it restriction} of $K$ to $M$.
  For $v \in [m]$ we denote by
  $K\setminus v := K|_{[m] \setminus \{v\}}$ the {\it deletion} of $v$ and
  by $\link_K(v) = \{ \sigma \in K ~|~v \not\in \sigma, \sigma \cup \{v\} \in K\}$ its {\it link}.
  We consider the link $\link_K(v)$ and the deletion $K \setminus v$
  of $K$ at a vertex $v$ as simplicial complexes over the ground set
  $[m]\setminus\{v\}$. For two simplicial complexes $K_1$ and $K_2$ over
  disjoint ground sets we write $K_1 \join K_2 := \{ A_1 \cup A_2~| A_i \in K_i, i=1,2\}$
  for their {\it join}.
  Using this terminology the (closed) {\it star} $\star_K(v) := \{ \sigma \in K ~|~\sigma \cup \{v\} \in K\}$
  of a vertex $v$ can be written as $\star_K(v)= \{ \emptyset, \{v\} \} \join \link_K(v)$.

  The notion of a vertex decomposable simplicial complex was defined
  for pure complexes in \cite{PB} and extended to nonpure
  complexes in \cite{BW}.
  The class of vertex decomposable simplicial complexes if defined inductively. A simplicial complex
  $K$ is called {\it vertex decomposable} if:

  \begin{itemize}
    \item[(i)]  $K$ is a simplex or $K=\{\emptyset\}$, or
    \item[(ii)] there exists a vertex $v$ such that:
    \begin{itemize}
      \item[(a)] $K\setminus v$ and $\link_K(v)$ are
        vertex-decomposable
      \item[(b)] no facet of $\link_K(v)$ is a facet of
        $K\setminus v$.
    \end{itemize}
  \end{itemize}

  \noindent The distinguished vertex $v$ in (ii) is called a {\it
  shedding vertex}. We say that an ordering $(v_1, v_2,\ldots,v_\ell)$
  of some of the vertices of a simplicial complex $K$ is a {\it shedding
  sequence} if $v_k$ is a shedding vertex of the consecutive
  deletion $K\setminus v_\ell\setminus\cdots\setminus v_{k+1}$ for all
  $k=1,2,\ldots,\ell-1$ and $K\setminus v_\ell\setminus\cdots\setminus v_{1}$
  is a simplex.

  \begin{Example}
    The simplicial complex $K \subseteq 2^{[4]}$ with facets $\{1,3\}$ and $\{2,4\}$ has
    the vertex decomposable Alexander dual $K^{\dual}$ with facets $\{1,2\}$,$\{1,4\}$,$\{2,3\}$,$\{3,4\}$.
    In this case $(1,2,3,4)$ is a shedding sequence for $K^{\dual}$, while the sequence $(1,3,2,4)$ is
    not (see Figure \ref{fig1}).
  \end{Example}

  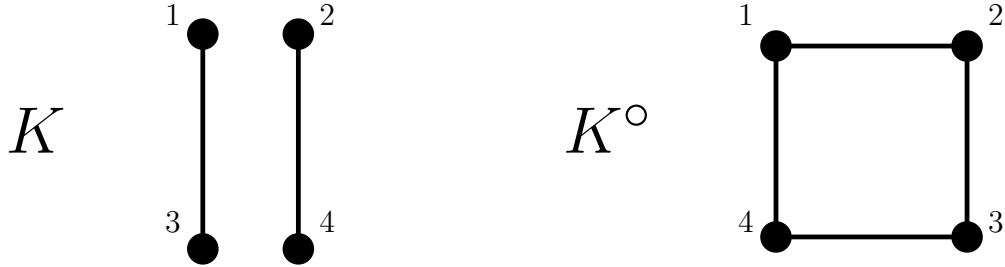
\begin{figure}
    \label{fig1}
    \vskip6cm
    \hskip-10cm
    \begin{picture}(0,0)%
      \includegraphics{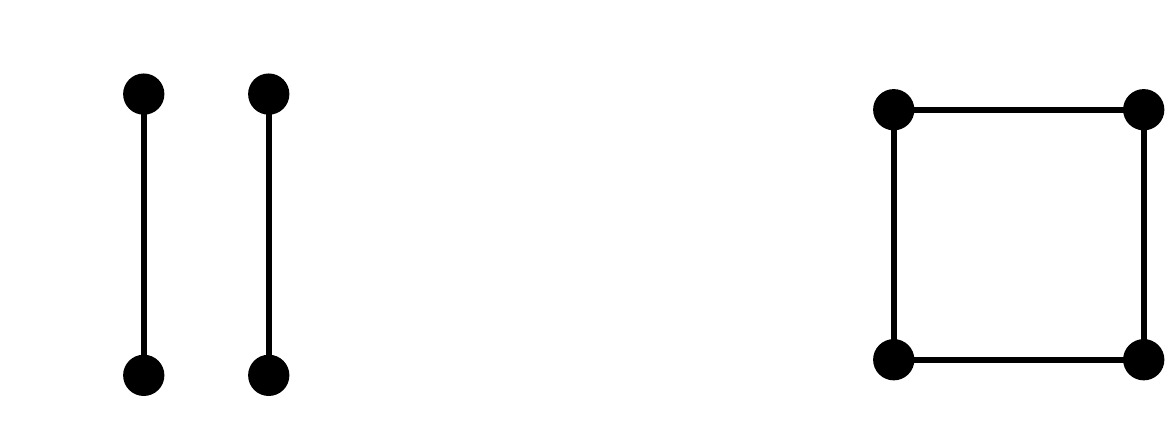}
    \end{picture}%

    \setlength{\unitlength}{3947sp}%
    \begingroup\makeatletter\ifx\SetFigFont\undefined%
      \gdef\SetFigFont#1#2#3#4#5{%
      \reset@font\fontsize{#1}{#2pt}%
      \fontfamily{#3}\fontseries{#4}\fontshape{#5}%
      \selectfont}%
    \fi\endgroup%
    \begin{picture}(5597,2115)(1111,-2551)
      \put(1000,  500){\makebox(0,0)[lb]{\smash{{\SetFigFont{12}{14.4}{\rmdefault}{\mddefault}{\updefault}{\color[rgb]{0,0,0}\huge{$K$}}}}}}
      \put(2000, 1300){\makebox(0,0)[lb]{\smash{{\SetFigFont{12}{14.4}{\rmdefault}{\mddefault}{\updefault}{\color[rgb]{0,0,0}$1$}}}}}
      \put(2970, 1300){\makebox(0,0)[lb]{\smash{{\SetFigFont{12}{14.4}{\rmdefault}{\mddefault}{\updefault}{\color[rgb]{0,0,0}$2$}}}}}
      \put(2000,    0){\makebox(0,0)[lb]{\smash{{\SetFigFont{12}{14.4}{\rmdefault}{\mddefault}{\updefault}{\color[rgb]{0,0,0}$3$}}}}}
      \put(2970,    0){\makebox(0,0)[lb]{\smash{{\SetFigFont{12}{14.4}{\rmdefault}{\mddefault}{\updefault}{\color[rgb]{0,0,0}$4$}}}}}

      \put(4500, 500){\makebox(0,0)[lb]{\smash{{\SetFigFont{12}{14.4}{\rmdefault}{\mddefault}{\updefault}{\color[rgb]{0,0,0}\huge{$K^\dual$}}}}}}
      \put(5600,1300){\makebox(0,0)[lb]{\smash{{\SetFigFont{12}{14.4}{\rmdefault}{\mddefault}{\updefault}{\color[rgb]{0,0,0}$1$}}}}}
      \put(7170,1300){\makebox(0,0)[lb]{\smash{{\SetFigFont{12}{14.4}{\rmdefault}{\mddefault}{\updefault}{\color[rgb]{0,0,0}$2$}}}}}
      \put(5600,   0){\makebox(0,0)[lb]{\smash{{\SetFigFont{12}{14.4}{\rmdefault}{\mddefault}{\updefault}{\color[rgb]{0,0,0}$4$}}}}}
      \put(7170,   0){\makebox(0,0)[lb]{\smash{{\SetFigFont{12}{14.4}{\rmdefault}{\mddefault}{\updefault}{\color[rgb]{0,0,0}$3$}}}}}
    \end{picture}%
    \vskip-4.5cm

    \caption{Simplicial Complex and its Dual}
  \end{figure}

  We are interested in the class of simplicial complexes with vertex
  decomposable Alexander duals. First we list a some well known facts
  that we will use in subsequent arguments.

  \begin{Remark}
    Let $K$ be a simplicial complex on ground set $[n]$.
    \label{rem:general}
     \begin{itemize}
       \item[(i)] Then $N \subseteq [n]$ is a minimal non-face of $K$ if
           and only if $[n] \setminus F$ is a facet if $K^\dual$.
       \item[(ii)] For $v \in [n]$ we have
           $\link_K(v)^{\dual}=K^{\dual}\setminus v$ and
           $(K\setminus v)^{\dual} = \link_{K^{\dual}}(v)$.
     \end{itemize}
  \end{Remark}

  The definition of vertex decomposability and \ref{rem:general} immediately imply the following
  lemma.

  \begin{Lemma}\label{l3}
    Let $K$ be a simplicial complex with vertex decomposable Alexander dual
    $K^{\dual}$. If $v$ is a shedding vertex of $K^{\dual}$ then the
    link $\link_K(v)$ and the deletion $K\setminus v$ have vertex
    decomposable duals.
  \end{Lemma}

  \subsection{Polyhedral Products} Let us first collect some basic facts about polyhedral products that we will use
  in subsequent sections without explicit reference.

  \begin{itemize}
    \item[$\diamond$] If $L\subset K$ is a subcomplex over the same ground set
      then $\zZ_L(X,A)$ a subspace of $\zZ_K(X,A)$.
    \item[$\diamond$] A map $f:(X,A)\longrightarrow (Y,B)$ gives rise to a
      map $\zZ_K(f):\zZ_K(X,A)\longrightarrow\zZ_K(Y,B)$.
      For the composition of induced maps we have
      $\zZ_K(g\circ f)=\zZ_K(g)\circ\zZ_K(f)$.
    \item[$\diamond$] $\zZ_(\bullet,\bullet)$ is the homotopy functor; that is the homotopy type of
      $\zZ_K(X,A)$ depends only on the relative homotopy type of
      the pair $(X,A)$.
  \end{itemize}

  Next we stress the relevance of the spaces $\zZ_K(D^{n}, S^{n-1})$ by
  exhibiting a homotopy equivalence to $\zZ_K(\RR^{n}, \RR^{n}\setminus\{0\})$,
  which in turn is easily seen to be the complement of a suitably chosen
  coordinate subspace arrangement in $(\RR^{n})^{m}$.
  Clearly, $\zZ_K(\CC, \CC\setminus\{0\}) = \zZ_K(\RR^2, \RR^2\setminus\{0\})$ is the complement of the
  vanishing locus of the Stanley-Reisner ideal of $K^\dual$ in $\CC^m$.

  Before we can show that $\zZ_K(D^{n}, S^{n-1})$ and $\zZ_K(\RR^{n}, \RR^{n}\setminus\{0\})$
  are homotopy equivalent we need one elementary set theoretic fact.

  \begin{Lemma}\label{l1}
    \[\zZ_{K^{\dual}}(X,A)=X^{m}\setminus\zZ_K(X,X\setminus A).\]
  \end{Lemma}

  \begin{proof}
    For a $x=(x_1,\ldots,x_m)\in X^{m}$ denote by $I_x=\{i\in[m]\mid
    x_i\in A\}$. We have that $x\in\zZ_{K^{\dual}}(X,A)$ if and only if
    $\sigma^{c}\subset I_x$ for some $\sigma\in K^{\dual}$ and
    $x\notin\zZ_K(X,X\setminus A)$ if and only if
    $I_x\cap\sigma^{c}\neq\emptyset$ for all $\sigma\in K$. Both
    statements are equivalent to $I_x\notin K$.
  \end{proof}

  The coordinate subspace arrangement $\aA_{K,n}$
  associated to the complex $K$ is defined as the $K^\dual$-power of
  the pair $(\RR^n,\{0\})$
  \[\aA_{K,n}=\zZ_{K^{\dual}}(\RR^{n},\{0\})=\bigcup_{\sigma\in
  K^{\dual}}(\RR^{n},\{0\})^{\sigma}.\] By \ref{l1}, the
  complement
  $\aA_{K,n}^c=(\RR^n)^m\setminus\aA_{K,n}$ of the
  arrangement $\aA_{K,n}$ is determined as the polyhedral
  product of the pair $(\RR^n,\RR^n\setminus\{0\})$
  \[\aA_{K,n}^c=\zZ_K(\RR^{n},\RR^{n}\setminus\{0\}).\]
  Since $(\RR^{n},\RR^{n}\setminus\{0\})$ and
  $(D^{n},D^{n}\setminus\{0\})$ are relative homotopy equivalent and
  $\zZ_K(X,A)$ is a homotopy functor, it
  follows that $\aA_{K,n}^c$ and $\zZ_K(D^{n},
  D^{n}\setminus\{0\})$ have the same homotopy type. This type of reasoning
  does not suffice to prove the homotopy equivalence of $\zZ(D^{n},D^{n}\setminus\{0\}))$ and $\zZ(D^n, S^{n-1})$
  since $(D^{n},D^{n}\setminus\{0\}))$ and $(D^n, S^{n-1})$ are not
  homotopy equivalent as pairs. Indeed Jelene Grbi\'s pointed out to us that $\zZ_K(\bullet,\bullet)$ is a homotopy
  functor in both arguments from which the results follows easily. For the sake of an explicit homotopy we still provide the
  proof. For $n=2$ an explicit homotopy equivalence of $\zZ(D^{2},D^{2}\setminus\{0\}))$ and $\zZ(D^2, S^1)$
  was established in \cite{BP}. We provide the homotopy equivalence for any $n \geq 1$.

  \begin{Proposition}\label{t2}
    For $n \geq 1$ the complement $\aA_{K,n}^c$ of the arrangement
    $\aA_{K,n}$ is homotopy equivalent to the $n$-sphere
    moment-angle complex $\zZ_K(D^{n},S^{n-1})$.
  \end{Proposition}
  \begin{proof}
    A simple retraction argument shows that
    the complex $\aA_{K,n}^c = \zZ_K(\RR^{n},\RR^{n}\setminus\{0\})$ is
    homotopy equivalent to $\zZ_K(D^{n},D^{n}\setminus\{0\})$.
    We inductively construct another deformation retraction
    $$ r_K : \zZ_K(D^n;D^n \setminus \{0\}) \rightarrow \zZ_K(D^n; S^{n-1}).$$

    Our induction basis is the case when $K=2^{[k]} \subseteq 2^{[m]}$ is a full $k$-simplex. In this case we have $\zZ_K(D^n;D^n \setminus \{0\}) =
    D^{nk} \times (D^n \setminus \{0\})^{m-k}$ and $\zZ_K(D^n; S^{n-1}) = D^{nk} \times (S^{n-1})^{m-k}$.
    Then the identity on the coordinates $[k]$ and the radial projection $\pi_n : D^n \setminus \{0\} \mapsto S^{n-1}$ in the coordinates
    $\{ k+1, \ldots ,m\}$ provides a deformation retraction from $r_K : \zZ_K(D^n;D^n \setminus \{0\}) \rightarrow \zZ_K(D^n; S^{n-1})$.

    In the induction step $K = K' \cup \{\tau\}$ for some facet $\tau$ of $K$. We may assume $\tau = [k]$ after suitable relabeling.
    In this situation
    \begin{eqnarray*}
       \zZ_{K'}(D^n;D^n \setminus \{0\})  & = & \zZ_{K}(D^n;D^n \setminus \{0\}) \setminus \{ (0,\cdots, 0) \times (D^{n} \setminus \{0\})^{m-k} \\
    \end{eqnarray*}
    and
    \begin{eqnarray*}
       \zZ_{K'}(D^n;S^{n-1}) & = & \zZ_{K}(D^n;S^{n-1}) \setminus (D^n \setminus S^{n-1})^k \times (S^{n-1})^{m-k} \\
    \end{eqnarray*}

    \noindent {\sf Case 1:} $K = 2^{[k]}$ is a full $(k-1)$-simplex and $\tau = [k]$.

    Then \begin{eqnarray*}
       \zZ_{K'}(D^n;D^n \setminus \{0\})  & = &  (D^{nk} \setminus \{ (0,\ldots,0)\}) \times (D^{n} \setminus \{0\})^{m-k}
    \end{eqnarray*}
    and
    \begin{eqnarray*}
       \zZ_{K'}(D^n;S^{n-1}) & = & S^{nk-1} \times (S^{n-1})^{m-k}.
    \end{eqnarray*}
    Hence the projection $D^{nk} \setminus \{ (0,\ldots,0)\}) \mapsto S^{nk-1}$ on the first $k$ coordinates and the
    projection $\pi_n$ in each of the remaining coordinates yields a deformation retraction $r_{K'}$.

    \noindent {\sf Case 2:} $\tau = [k]$ is not the unique facet of $K$.

    In this case we can write $K$ as the union of $K'$ and $2^{[k]}$ with intersection $L = 2^{[k]} \setminus \{[k]\}$.
    By induction for $K$ and by Case 1 for $L$ we have deformation retraction
    maps $r_{K} : \zZ_{K}(D^n;D^n \setminus \{0\}) \rightarrow \zZ_{K}(D^n; S^{n-1})$ and
    $r_{L} : \zZ_{L}(D^n;D^n \setminus \{0\}) \rightarrow \zZ_{L}(D^n; S^{n-1})$.
    Let $i : \zZ_{K'}(D^n;D^n \setminus \{0\}) \hookrightarrow \zZ_{K}(D^n;D^n \setminus \{0\})$ be the inclusion map.
    Then consider $r:=r_{L} \circ r_{K} \circ i$. Then the image of $r$ is $\zZ_{K'}(D^n; S^{n-1})$ and a simple computation shows that
    is $r_{K'} = r$ is the desired retraction.
  \end{proof}

  Complements of subspace arrangements are an interesting object for their own sake and their homotopy
  types are usually difficult to approach.

  We finally recall a result from \cite{BBCG} on the homotopy
  type of the suspension $\Sigma \zZ_K(D^n,S^{n-1})$ of $\zZ_K(D^n,S^{n-1})$, that by \ref{t2} also
  gives a homotopy decomposition of the suspension of the subspace arrangement complement $\aA_{K,n}^c$.

  \begin{Proposition}[Cor. 2.24 in \cite{BBCG}]
     \label{pr:susp}
     For $n \geq 2$ there is a homotopy equivalence
     $$\Sigma \zZ_K(D^n,S^{n-1}) \simeq  \bigvee_{{{M \not\in K} \atop {M \subseteq [m]}}} \Sigma K_M \join S^{(n-1)\# M}.$$
  \end{Proposition}

\section{Discrete Morse theory for $\zZ_K(D^{n}, S^{n-1})$}
  \label{sec:dimo}

  In this section we exhibit the tools and perform the basic constructions
  needed for the proof of \ref{t1}.

  Let $X$ be a compact regular CW-complex. It is well known
  \cite[Prop. 1.2.]{FP} that for regular CW-complexes the topology
  of $X$ is determined by its face poset; that is the partial order
  $\preceq$ on the set $X^{(\cells)}$ of closed cells of $X$ ordered
  by $c\preceq c'$ if and only if $c\subseteq c'$. The directed
  graph $G_X=(V_X, E_X)$ on vertex set $V_X=X^{(\cells)}$ and edge set
  \[E_X=\{c\rightarrow c'\mid c, c'\in X^{(\cells)}, c\succeq c',
  \dim(c)=\dim(c')+1\}\] is the graph of the Hasse diagram of $X^{(\cells)}$.
  An acyclic matching on $G_X$ is a set $ \mM \subseteq E_X$
  of edges of $G_X$ such that:

  \begin{itemize}
    \item[$\diamond$] Each cell from $X^{(\cells)}$ appears in at most once
       edge from $\mM $
    \item[$\diamond$] The directed graph $G_X^{ \mM}=(V_X, E_X^{\mM})$
       with edge set
       $E_X^{ \mM}=E_X\backslash \mM\cup\{c'\rightarrow c\mid c\rightarrow c'\in \mM\}$
       contains no directed cycles.
  \end{itemize}

  Note that $G_X^{\mM}$ arises from $G_X$ by reversing all
  edges from $\mM$. Discrete Morse theory was developed by
  Forman \cite{F} in order to explicitly determine the homotopy
  type of regular CW-complexes. In order to formulate his main
  result we need some more notation. For an acyclic matching $
  \mM $ in $G_X$ we call a cell $c\in X^{(\cells)}$ an $\mM$-critical
  cell if $c$ does not appear in any edge from $\mM$.
  We write $\Crit(\mM)$ for the critical cells of the matching $\mM$.
  For an $i$-cell
  $c$ denote by $f_c$ the attaching map
  $f_c:\partial c\longrightarrow X^{i-1}$, where $X^{i-1}$ is
  the $(i-1)$-skeleton of $X$. To any edge $e:c\rightarrow c'$ in
  $G_X^{ \mM}$ we associate a map $f_e$. If $c\rightarrow
  c'$ is an edge from $E_X$ then this map is the restriction
  $f_c\mid_{f_c^{-1}(c')}$ of the attaching map $f_c$ of $c$ to the
  preimage of $c'$. If $e:c\rightarrow c'\in E_X^{\mM}\backslash E_X$
  then $f_e$ is an arbitrary but fixed homeomorphism from $c$ to
  $\partial c'$ for which $f_{c'} \circ f_{e} (x) = x$ for all $x \in \partial c$. For two $ \mM-$critical
  cells $c$ and $c'$ we call a directed path
  $p:c=c_0\rightarrow\cdots\rightarrow c_n=c'$ in
  $G_X^{ \mM}$ of length $n\geq 1$ a gradient path. To each
  gradient path we associate the map $f_p:\partial c\longrightarrow c'$
  which is the composition of the maps $f_e$ corresponding to
  the edges of $p$. It is easy to see that the map $f_p$ for all
  gradient paths starting in the $i$-critical cell $c$ define a map
  $f_c^{ \mM}$ from $\partial c$ to the set of $\mM$-critical cells of
  dimension less than $i$. By induction one then shows that this defines
  a CW-structure on the set $\Crit(\mM)$ of $\mM$-critical cells. We denote this
  CW-complex by $X^{\mM}$. We are now in position to formulate the main
  result from Forman's discrete Morse theory using its reformulation
  in terms of acyclic matchings provided by Chari \cite[Proposition 3.3]{C}.

  \begin{Theorem}[Forman, Theorem 3.4 in \cite{F}] \label{t3}
    Let $X$ be a compact
    regular CW-complex and let $ \mM$ be an acyclic matching
    on $G_X$. Then $X$ is homotopy equivalent to the CW-complex
    $X^{\mM}$, where $\mM$-critical cell $c \in \Crit(\mM)$ of dimension $i$ is
    attached in $X^{\mM}$ to the $(i-1)$-skeleton of $X^{\mM}$
    via the map $f_c^{\mM}$.
  \end{Theorem}

  Let $(X,A)$ be a regular CW-pair with the set of cells $V_A\subseteq
  V_X$, where $V_X=\{c_\lambda\}_{\lambda\in I}$ and
  $V_A=\{c_\lambda\}_{\lambda\in J}$ for some $J\subseteq I$. Then
  $\zZ_K(X,A)$ is the CW-subcomplex of the product $X^{m}$
  with the set of cells

  \[V_{\zZ_K(X,A)}=\bigcup_{\sigma\in K}
    \{c_{\lambda_1}\times\cdots\times c_{\lambda_m}\mid\lambda_i\in
    J, i\notin\sigma\}.
  \]

  Since any product of regular CW-complexes is regular \cite[Theorem 2.2.2 (iv)]{FP}
  it follows that $X^m$ carries a regular CW-structure and so does any
  subcomplex. In particular, the cellular structure on $\zZ_K(X,A)$ is regular.
  The directed graph $G_{\zZ_K(X,A)}$
  is the subgraph of the product graph $G_{X^{m}}=(G_X)^{m}$.

  In the following for a simplicial complex $K$ and assuming a fixed pair $(X,A)$ we write
  $V_K$ for $V_{\zZ_K(X,A)}$, $E_K$ for $E_{\zZ_K(X,A)}$ and $G_K$ for $G_{\zZ_K(X,A))}$.

  The next
  lemma follows immediately from definitions.

  \begin{Lemma}\label{l2}
    If $K_M$ is the restriction of the simplicial complex $K$ on a
    subset $M\subseteq [m]$, then
    \[\zZ_{K_M}(X,A)=(X,A)^{M}\cap\zZ_K(X,A).\]
  \end{Lemma}

  In the following, we assume that $(X,A)=(D^{n}, S^{n-1})$ for some
  $n\geq 1$, with the minimal regular cellular structure. That is,
  the sphere $S^{n-1}$ has two cells $e_-^i,e_+^i$ in each dimension $i$
  for $0 \leq i \leq n-1$ and the disc $D^{n}$ has one additional cell $e_\ncell^n$
  of dimension $n$. Thus
  we have $V_{S^{n-1}}=\{e_-^{i}, e_+^{i}\mid i=0,1,\ldots, n-1\}$ and
  $V_{D^{n}}=V_{S^{n-1}}\cup\{e_\ncell^n\}$. Let $G$ be the graph of the
  corresponding Hasse diagram of the disc $D^{n}$ and $\mM$
  be the acyclic matching on $G$ defined by

  \[
    \mM=\{e_-^{i+1}\rightarrow e_+^{i}\mid i=0,1,\ldots
    n-2\}\cup\{ e_\ncell^n \rightarrow e_+^{n-1}\}.
  \]

  The graph $G^\mM$ is visualized in Figure \ref{fimatch}.

  \begin{figure}
    \begin{diagram}
      e^0_{-}       & \lTo      & e^1_{-} & \lTo       & \cdots & \lTo      & e^{n-1}_{-} &              & \\
                    &\luTo\ruTo &         & \luTo\ruTo & \cdots &\luTo\ruTo &             & \luTo(2,1)  & e_\ncell^n \\
      e^0_{+}       & \lTo      & e^1_{+} & \lTo       & \cdots & \lTo      & e^{n-1}_{+} & \ruTo(2,1)
    \end{diagram}
    \caption{The graph $G^\mM$}
    \label{fimatch}
  \end{figure}

  \noindent We call the cells $e_-^i$ minus-cells and the
  cells $e_+^i$ plus-cells and consider minus or plus as their signs.
  Let $G_m=(V_m, E_m)$ be the graph of the Hasse diagram
  of the cell decomposition induced on the product $(D^{n})^{m}$ with the vertex set
  $V_m=\{c_1\times\cdots\times c_m\mid c_i\in V_{D^{n}}\}$. For a
  cell $c=c_1\times\cdots\times c_m\in V_m$, define
  $\supp(c)=\{i\in[m]\mid c_i=e_\ncell^n\}$. Let $K$ be an arbitrary
  simplicial complex $K$ over the ground set $[m]$. The graph $G_K$
  is the subgraph of $G_m$ with the vertex set $V_K=\{c\in V_m\mid
  \supp(c)\in K\}$. We construct the matching $\mM_K$
  inductively. Let

  \[
    \mM_1=\{c\rightarrow c'\in E_K\mid
    c_1\rightarrow c'_1\in\mM\},
  \]
  \[
    \mM_{k+1}=\{c\rightarrow c'\in E_K\mid c,c'\in
    \Crit(\bigcup_{i=1}^{k}\mM_i), c_{k+1}\rightarrow
    c'_{k+1}\in\mM\},
  \]
  for $k=1,\ldots,m-1$. Define $\mM_K$ as the union
  $\mM_K=\bigcup_{k=1}^{m}\mM_k$. By construction,
  $\mM_K$ is a matching.

  \begin{Proposition}\label{p1}
    The matching $\mM_K$ is acyclic.
  \end{Proposition}
  \begin{proof}
    For each cell $c = c_1 \times \cdots \times c_m$ we denote by
    $\ell(c)$ the sum of the dimension of the $c_i$ plus the
    number of indices $i$ for which $c_i$ is a plus-cell.

    \noindent {\sf Claim:} The function $\ell(\cdot)$ is weakly
    decreasing on a directed path in $G^\mM$.
    \begin{proof}[Proof of Claim]
      Let $e = c \rightarrow c'$ be an edge in $G^\mM$ for
      $c = c_1 \times \cdots \times c_m$ and $c' = c_1' \times \cdots \times c_m'$.
      Then there is a unique index $i$ for which $c_i \neq c_i'$.

      \noindent {\sf Case 1:} $e$ is an edge from $G$. Then we have
      $\dim c_i = \dim c_i'+1$. Hence $\ell(c) = \ell(c') +1$
      if the signs of $c_i$ and $c_i'$ coincide and $\ell(c) = \ell(c')+2$ if the
      sign of $c_i$ is positive and the sign of $c_i'$ is negative.

      \noindent {\sf Case 2:} $e$ is not an edge from $G$. Then we have
      $\dim c_i = \dim c_i'-1$ and $c_i$ is a plus-cell and $c_i'$ is either a minus-cell or
      a $e_\ncell^n$.
      Thus $\ell(c) = \ell(c')$.
   \end{proof}

   Thus on a directed cycle
   $p:c^{1}\rightarrow\cdots\rightarrow c^{r} = c^{1}$ we must have
   $\ell(c^{1}) = \cdots = \ell(c^{r-1})$. From the proof of the claim it
   follows that on a directed cycle one can have only edges $e$ which increase the
   dimension of the cells. But this cannot lead to a directed cycle and hence
   the matching is acyclic.
  \end{proof}

  Since from now on we will only consider the matching $\mM_K$ defined above
  we write $\Crit(K)$ for $\Crit(\mM_K)$. In addition,
  we write $\Crit_k$ for  $\Crit(K|_{[k]}),k=1,2,\ldots m$ and
  $\Crit_{k\ncell}$ for $\Crit(\link_{K|_{[k+1]}}(k+1))$.
  The
  following proposition gives an inductive construction of the set
  of $\mM_K$-critical cells.
  We use the following convention, if $\cC$ is a set of cells then we denote
  by $\cC^-$ the set of cells $c \times e_-^0$ for $c \in \cC$, by $\cC^+$ the set of
  cells $c \times e_+^{n-1}$ for $c \in \cC$ and by $\cC^\ncell$ the set of cells
  $c \times e_\ncell^n$ for $c \in \cC$.

  \begin{Proposition}\label{p2}
    \begin{itemize}
      \item[(i)]
        $$\Crit_1=\left\{\begin{array}{cc}\{e_-^0 \}, & \{1\}\in K \\
            \{e_-^0,e_+^{n-1}\}, & \{1\}\notin K\end{array}\right..$$
      \item[(ii)]  For $k \in [m-1]$
        \[ \Crit_{k+1} =
          \Crit_k^- \cup \Crit_k^+ \cup \Crit_{k\ncell}^\ncell \setminus
          \Big( \Crit_k^\ncell \cup \Crit_{k\ncell}^+ \Big).
        \]
        In particular, any critical cell is a product of cells
        from the set $\{ e_-^0,e_+^{n-1},e_\ncell^n\}$.
    \end{itemize}
  \end{Proposition}
  \begin{proof}
    The first claim follows from a simple inspection.

    Let $k \in [m-1]$ and $c_1\times\cdots\times c_{k+1} \in \Crit_{k+1}$ be a
    critical cell. Set $c = c_1 \times \cdots \times c_k$.
    \begin{itemize}
      \item[$\diamond$] $c_{k+1} \neq e_\ncell^n$: Then by our inductive construction it follows that
         $c \in \Crit_k$.
      \item[$\diamond$] $c_{k+1} = e_\ncell^n$. Then again the inductive construction implies that
         any matching of cells in $\zZ_{K|_{[k]}}(D^n,S^{n-1})$ that is not induced by
         matching a cell $c_i = e_{+}^{n-1}$ for some $1\leq i \leq k$ with $e_\ncell^n$ can also be applied to
         to $c \times c_{k+1}$.
    \end{itemize}
    It follows that $c_i \in \{ e_-^0,e_+^{n-1},e_\ncell^n \}$ for $1 \leq i \leq k$.

    Assume that $c_{k+1} \not\in \{ e_-^0,e_+^{n-1},e_\ncell^n \}$. Then there is
    a cell $c_{k+1}' \not\in  \{ e_-^0,e_+^{n-1},e_\ncell^n \}$ such that
    $c_{k+1}$ and $c_{k+1}'$ are matched in $\mM$.
    Since $\supp (c \times c_{k+1}) = \supp (c \times c_{k+1}')$ it follows that $c' \times
    c_{k+1}'$ is a cell from $\zZ_{K|_{[k+1]}}(D^n,S^{n-1})$.
    By the inductive
    nature of our matching the cells $c \times c_{k+1}$ and $c \times = c_1 \times \cdots \times c_k$
    are matched in $\mM_{K|_{[k+1]}}$.
    But this contradicts $c \times c_{k+1} \in \Crit_{k+1}$.
    Hence we have $c_{k+1} \in \{ e_-^0,e_+^{n-1},e_\ncell^n \}$ and we distinguish the three cases.

    \begin{itemize}
      \item[$\diamond$] $c_{k+1} = e_-^0$.

        From $c \in \Crit_k$ and the nature of $\mM$ it follows that the cell $c \times c_{k+1}$ is
        not matched in $\mM_K$. In particular, we have $c \times c_{k+1} \in \Crit_k^- \subseteq \Crit_{k+1}$.

      \item[$\diamond$] $c_{k+1} = e_+^{n-1}$.

        Then from $c \in \Crit_k$ it follows that $\supp(c) \cup \{k+1\} \not\in K|_{[k+1]}$.
        Hence $\supp(c) \not\in \link_{K|_{[k+1]}}(k+1)$. This implies $c \times c_{k+1} \in
        \Crit_k^+ \setminus \Crit_{k\ncell}^+ \subseteq \Crit_{k+1}$.

      \item[$\diamond$] $c_{k+1} = e_\ncell^n$.

        Then for $c_{k+1}' = e_+^{n-1}$ we have $c \times c_{k+1}'$ is a cell in
        $\zZ_{K|_{[k+1]}}(D^n,S^{n-1})$. Thus by definition of $\mM$ in order to have $c \times c_{k+1} \in \Crit_{k+1}$
        we need $c \times c_{k+1}' \not\in \Crit_{k+1}$. The latter implies $c \not\in \Crit_k$.
        Hence $c \times c_{k+1} \in \Crit_{k\ncell}^\ncell\setminus \Crit_k^\ncell \subseteq \Crit_{k+1}$.
    \end{itemize}
    Summarizing the three cases and observing that the set involved in any two of the three cases are
    mutually disjoint implies the assertion.
  \end{proof}

  By \ref{p2} we know that if $c_1 \times \cdots \times c_m$ is a critical cell then
  $c_i \in  \{e_-^0,e_+^{n-1},e_\ncell^n\}$ for $1 \leq i \leq m$.
  Hence from now on we can identify a critical cell with an $m$-tuple in $\{+,-,\ncell\}^m$.
  For a cell $c=c_1\times\cdots\times c_m\in V_m$ and $\sgn \in \{ +,-,\ncell\}$
  we define $c(\sgn) = -\infty$ if $\sgn \neq c_i$ for $1 \leq i \leq m$ and $c(\sgn)=\min\{i\mid c_i=\sgn\}$
  otherwise. The following is a simple corollary of \ref{p2}.

  \begin{Corollary}
    \label{c1}
    Let $c = c_1 \times \cdots \times c_m \in \Crit(K)$ be a critical cell and for $1 \leq i \leq m$ set $A_i = \supp(c) \cap \{i+1,\ldots, m\}$.  Then:
    \begin{itemize}
      \item[(i)] For $1 \leq i \leq m$ we have $c_1 \times \cdots \times c_i \in \Crit(\link_K(A_i)|_{[i]})$. If $i \in \supp(c)$ then
           $c_1 \times \cdots \times c_{i-1} \not\in \Crit(\link_K(A_i)|_{[i-1]})$.
      \item[(ii)] If $c(\ncell)  \neq -\infty$ then
         \begin{itemize}
             \item[(a)] $-\infty < c(+) < c(\ncell)$ and
             \item[(b)] if $c_i = \ncell$ for some $1 \leq i \leq m$ then
                there exists an $1 \leq j < i$ such that $c_j = +$, $\supp(c) \cup \{j\} \not\in K$ and
                $\supp(c) \setminus \{ i\} \cup \{j\} \in K$.
          \end{itemize}
      \item[(iii)] If $c(+) = - \infty$ then $c = (-,\ldots,-)$.
      \item[(iv)] If $c(+) \neq - \infty$ and $c_i = +$ for some $1 \leq i \leq m$ then we have $\supp(c) \cup \{i\} \not\in K$.
    \end{itemize}
  \end{Corollary}
  \begin{proof}
    \begin{itemize}
      \item[(i)]
         Suppose there is a cell $c_1' \times \cdots \times c_i'$ matched with $c_1 \times \cdots \times c_i$ in
         $\mM_{\link_K(A_i)|_{[i]}}$ then by the inductive construction of the matching, the
         cells $c$ and $c' = c_1' \times \cdots \times c_i' \times c_{i+1} \times \cdots \times c_m$
         are matched in $\mM_K$. But this contradicts the assumption that $c \in \Crit(K)$.
         If $c_i = \ncell$ then it is an immediate consequence of \ref{p2} that $c_1 \times \cdots \times c_{i-1} \not\in \Crit(\link_K(A_i)|_{[i-1]})$.
       \item[(iv)] Suppose there is an $1 \leq i \leq m$ such that $c_i = +$ and $\supp(c) \cup \{i\} \in K$.
         Then the cells $c_1 \times \cdots \times c_{i}$ and $ c_1 \times \cdots \times c_{i-1} \times \ncell$ are matched in $\mM_{\link_K(A_i)|_{[i]}}$
         which contradicts (i).
       \item[(ii)]
         (a) Let $i$ be $c(\ncell)$. Then (i) implies that $c_1 \times \cdots \times c_{i-1} \not\in \Crit(\link_K(A_i)|_{[i-1]}))$.
         Thus $c_1 \times \cdots \times c_{i-1} \neq (-, \ldots, -)$. Therefore by definition of $i = c(\ncell)$ it follows
         that there is an index $1  \leq j \leq i-1$ such that $c_j = +$. In particular, we have
         $-\infty < c(+) < i = c(\ncell)$.

         (b) Let $i$ be such that $c_i = \ncell$. Then by (i) $c' := c_1 \times \cdots \times c_{i-1}$ is not a critical cell of
         $\mM_{\link_K(A_i)|_{[i-1]}}$. Thus for some $1 \leq j < i$ such that $c_j = +$ we have
         $\supp(c)  \setminus \{i\} \cup \{j\} \in K$. The second assertion follows from (iv)
      \item[(iii)]
         If $c(+) = -\infty$ then it follows from (ii)(a) that $c(\ncell) = -\infty$ and hence
         $c = (-,\ldots, -)$.
    \end{itemize}
  \end{proof}

  \begin{Example}
    Let $K$ be the $4$-gon with set of facets $\{13,14,23,24\}$. Then the set of
    $\mM_K$-critical cells is equal to
    $\Crit(K)=\{(-,-,-,-), (-,-,+,\ncell), (+,\ncell, -,-),(+,\ncell, +,\ncell)\},$ which
    form $\zZ_K(D^{n},S^{n-1})=S^{2n-1}\times S^{2n-1}$.
  \end{Example}

  \begin{Lemma}
     \label{lem:simplex}
     Let $K$ be a simplicial complex on ground set $[m]$ and
     $c = c_1 \times \cdots \times c_m \in \mM(K)$ a cell such that
     $c(+) \neq -\infty$ and $c_i = +$ if and only if $i = c(+)$.
     Let $I = \{ c(+) \} \cup \supp(c)$.
     Then $c \in \Crit(K)$ if and only if $c_I \in \Crit(K_I)$ and $K_I
     = 2^I \setminus \{I\}$.
  \end{Lemma}
  \begin{proof}
     By \ref{c1} (ii)(a) we can assume that $I = \{ c(+) < i_1 <
     \ldots < i_k \}$.

     \begin{itemize}
       \item[$\Rightarrow$]
     If $c \in \Crit(K)$ then by \ref{c1} (ii)(b) $I \not\in K$
     and $I \setminus \{ i_j \} \cup \{ c(+) \} \in K$
     for $1 \leq j \leq k$.
     Since $\supp(c) \in K$ be definition we have that
     $K_I = 2^I \setminus \{I\}$. A simple check shows that
     $c_I$ is critical for $K_I$.
       \item[$\Leftarrow$]
    If $K_I =  2^I \setminus \{I\}$.
    Distinguish $\# I = 1$ and $\# I > 1$. $\# I = 1$
    then $I = \{c(+)\}$ and by \ref{p2} $c$ is critical in $K_I$ if and only if
    $I \not \in K$. Thus $K_I = \{ \emptyset \} = 2^I \setminus \{I\}$ and
    again by \ref{p2} $c$ is critical in $K$.
    Now let $\# I \geq 1$.
    Assume that $c \not\in \Crit(K)$. Then by \ref{p2} (ii) the cell
    $c = c_1 \times \cdots \times c_{i_k-1}$ either lies in $\Crit_{i_k-1}$ or
    does not lie $\Crit(\link_K(i_k))$. In the first case $I \setminus \{ i_k \}
    \not\in K$. Which contradicts $K_I = 2^I \setminus \{I\}$. In the second case
    we must have $I \setminus \{i_k\} \in \link_K(i_k)$ and hence $I \in K$.
    Again this contradicts $K_I = 2^I \setminus \{I\}$.
    Thus $c \in \Crit(K)$ and we are done.
    \end{itemize}
  \end{proof}

\section{Proof of \ref{t1} and \ref{co:susp}}
  \label{sec:proof}

  Let $K$ be a simplicial complex over the ground set $[m]$ with
  vertex decomposable Alexander dual $K^\dual$, such that $\{i\}\in K$ for all
  $i\in[m]$ and $m$ is a shedding vertex. By \ref{t3}, the $n$-sphere
  moment-angle complex $\zZ_K(D^n,S^{n-1})$ is homotopy equivalent to the
  space $\zZ_K(D^n,S^{n-1})^{\mM_K}$. We need to see how
  $\mM_K$-critical cells are glued to form the
  space $\zZ_K(D^n,S^{n-1})^{\mM_K}$. We need some preparatory lemmas.

  For a cell $c = c_1 \times \cdots \times c_m$ define
  $$J(c) := \left\{ \begin{array}{cc} -\infty                           & \mbox{~if~} c(+) = -\infty \\
                                      \max\{ i ~|~ c_i = +\}            & \mbox{~if~} c(\ncell) = -\infty < c(+) \\
                                      \max\{ i < c(\ncell)~|~ c_i = +\} & \mbox{~if~} -\infty \neq c(\ncell)
                    \end{array} \right.$$

  \begin{Proposition}
    \label{pr:crucial}
    Let $K$ be a simplicial complex on vertex set $[m]$ with vertex
    decomposable Alexander dual $K^\dual$. If $c = c_1 \times \cdots \times c_m \in \Crit(K)$
    then $\tilde{c} = \tilde{c_1} \times \cdots \times \tilde{c_m}
    \in \Crit(K)$ for
    $$\tilde{c_i} = \left\{ {{c_i \mbox{~if~} c_i \in \{ -,\ncell\} \mbox{~or~} i = J(c)} \atop
                             {-   \mbox{~if~} c_i = + \mbox{~and~} i \neq J(c)}}. \right.$$
  \end{Proposition}
  \begin{proof}
    We proceed by induction on $m$.
    If $m =1$ then either $K = \{ \emptyset \}$ or $K = \{ \emptyset , \{1\}\}$. In the
    first case $\Crit(K) = \{ -,+\}$ and for $c \in \Crit(K)$ we have $c = \tilde{c}$.
    In the second case $\Crit(K) = \{-\}$ and again for $c  = \tilde{c}$ for $c \in \Crit(K)$.

    Assume $m \geq 2$. In the following for $c \in \Crit(K)$ for which $c(\ncell) \neq -\infty$ we
    set $m(c) := \max \{ i ~|~c_i = \ncell\}$.
    We distinguish two cases:
    \begin{itemize}
       \item[$\rightarrow$]  $c(\ncell) = - \infty$

         This condition is equivalent to $\supp(c) = \emptyset$. Hence either $\tilde{c} = (-, \ldots,-)$
         or $\tilde{c} = (-,\ldots,-,+,-,\ldots,-)$ with $+$ in
         position $J(c)$. In the latter case we argue as follows.
         In the first case we clearly have $\tilde{c} \in \Crit(K)$. Assume the second case.
         Since $c$ is critical by \ref{p2} we must have that
         $\{ J(c) \} \not\in K$. The latter implies that $\tilde{c}$ is critical as well.

       \item[$\rightarrow$] $c(\ncell) \neq -\infty$ and $m(c) = m$.

         For this direction we make use of the notation used in \ref{p2}

         In this case \ref{p2} implies that $(c_1, \ldots, c_{m-1}) \in  \Crit(\link_K(m))$.
         By induction it follows that $\tilde{c}_1 \times \cdots \times \tilde{c}_{m-1}\in \Crit(\link_K(m))$
         and hence $\tilde{c} \in \Crit_{m-1\ncell}^\ncell$.
         The construction of $\tilde{c}$ allows to apply \ref{lem:simplex} which shows that
         $N := \{ J(c) \} \cup \supp(c) \setminus \{ m\} = \{ J(\tilde{c})\} \cup \supp(\tilde{c}) \setminus \{m\}$
         is a minimal nonface of $\link_K(m)$.

         \noindent {\sf Assumption:} $\tilde{c} \in \Crit_{m-1}^\ncell$

         \noindent $\triangleleft$ We infer that $(\tilde{c}_1, \ldots, \tilde{c}_{m-1}) \in \Crit(K \setminus m)$.
         Again by \ref{lem:simplex} we have that $N$ is a minimal non-face of $K \setminus m$.
         Thus $N$ is a minimal nonface of $\link_K(m)$ that is also a minimal non-face of $K \setminus m$.
         This shows that $[m-1] \setminus N$ is a facet of $K^\dual \setminus m$ and $\link_{K^\dual}(m)$.
         But this contradicts the fact that $m$ is a shedding vertex of $K^\dual$. Hence we have
         deduced a contradiction and the assumption is false.
         $\triangleright$

         Now we know that $\tilde{c} \not\in \Crit_{m-1}^\ncell$ and $\tilde{c} \in \Crit_{m-1\ncell}^\ncell$.
         Again \ref{p2} shows that $\tilde{c} \in \Crit(K)$.

       \item[$\rightarrow$] $c(\ncell) \neq -\infty$ and $m(c) < m$.
         In this case by \ref{p2} we have that $c_1 \times \cdots \times c_{m(c)} \in \Crit(K \setminus m \setminus \cdots \setminus m(c)+1)$.
         Since \ref{p2} also shows that $\tilde{c} \in \Crit(K)$ if and only if
         $\tilde{c}_1 \times \cdots \tilde{c}_{m(c)} \in \Crit(K \setminus m \setminus \cdots \setminus m(c)+1)$ the assertion
         follows by induction.
    \end{itemize}
  \end{proof}

  \begin{Corollary}
     \label{cor:nonface}
     Let  $K$ be a simplicial complex with vertex decomposable Alexander dual $K^\dual$.
     Then for $c \in \Crit(K)$ with $c(\ncell) \neq -\infty$ the set $\{ J(c) \} \cup \supp(c)$ is a minimal non-face of $K$.
  \end{Corollary}
  \begin{proof}
     Let $\tilde{c}$ be as in \ref{pr:crucial}. Then \ref{pr:crucial} implies that $\tilde{c} \in \Crit(K)$.
     The construction of $\tilde{c}$ assures that we can apply \ref{lem:simplex}. Hence $J(\tilde{c}) \cup \supp(\tilde{c})$ is a
     minimal nonface of $K$. Since $J(c) = J(\tilde{c})$ and $\supp(c) = \supp(\tilde{c})$ by the construction of
     $\tilde{c}$ the assertion follows.
  \end{proof}

  \begin{Proposition}
    \label{pr:trivial}
    Let $K$ be a simplicial complex on ground set $[m]$ with
    vertex decomposable Alexander dual $K^\dual$ and let $m$ be a shedding vertex for $K^\dual$.
    Then for any $n \geq 1$ the inclusion
    $i : \zZ_{\link_K(m)}(D^n,S^{n-1}) \hookrightarrow \zZ_{K \setminus m}(D^n,S^{n-1})$ is
    homotopically trivial.
  \end{Proposition}
  \begin{proof}
    It is sufficient to find a contractible CW-subcomplex $Y$ of the complex
    $\zZ_{K \setminus m}$, which contains $\zZ_{\link_K(m)}(D^n,S^{n-1})$.
    Let $c \in \Crit(\link_K(m))$. From \ref{cor:nonface} we
    know that $N:= \{ J(c) \} \cup \supp(c)$ is a minimal nonface of $\link_K(m)$.
    Thus by \ref{rem:general} $[m-1] \setminus N$ is a facet of
    $(\link_K(m))^\dual = K^\dual \setminus m$.
    Since $m$ is a shedding vertex of $K^\dual$ it follows that $[m-1] \setminus N$
    is not a face of $\link_{K^\dual} (m) = (K\setminus m)^\dual$.
    Therefore $\{J(c)\} \cup \{ \supp(c)\}$ must be a face of $K \setminus m$.
    To $c = c_1 \times \cdots \times c_{m-1} \in \Crit(\link_K(m))$ we associate the
    cell $c^\ncell = c_1^\ncell \times \cdots \times c_{m-1}^\ncell$ defined by $c_i^\ncell = c_i$ if
    $i \neq J(c)$ and $c_i^\ncell = \ncell$ if $i = J(c)$.
    Consider
    $$Y := \zZ_{\link_K(m)} \cup \Big\{ c^\ncell~|~c \in \Crit(\link_K(m)) \setminus \{ (-,\ldots,-)\}\Big\}.$$
    Then it is easily checked that $Y$ indeed is a subcomplex of $\zZ_{K \setminus m} (D^n,S^{n-1})$.
    By construction the edge $c^\ncell \rightarrow c$ is in $\mM_Y$ for any critical cell $c \in \Crit(\link_K(m))$.
    Since $\mM_{\link_K(m)} \subseteq \mM_Y$ it follows that $(-,\ldots,-)$ is the only critical cell in $Y$ and
    hence $Y$ is contractible.
  \end{proof}

  Before we can proceed to the proof of \ref{t1} we have to recall some
  basic facts from homotopy theory. For the sake of completeness we provide
  proofs for facts we could not find an explicit reference for.

  The first lemma is an elementary exercise in homotopy theory.

  \begin{Lemma}\label{l5}
    \begin{itemize}
      \item[(i)] If the pairs of CW-complexes $(X,A)$ and $(X',A')$ are relative
        homotopy equivalent, then $X/A$ and $X'/A'$ are homotopy equivalent.
      \item[(ii)] The following homotopy equivalence holds:
        \[S^{p}\times S^{q}/\pt\times S^{q}\simeq S^{p}\vee S^{p+q}.\]
      \item[(iii)] For any CW-complex $X$ the quotient $X \times D^n / X \times \pt$
        is contractible.
     \end{itemize}
  \end{Lemma}
  \begin{proof}
    \begin{itemize}
      \item[(i)]
        Let $h:(X,A)\longrightarrow (X',A')$ be a relative homotopy
        equivalence. Then, the following diagram is commutative:

        \[\begin{array}{ccccc}
          pt&\longleftarrow& A &\stackrel{i}\longrightarrow&X \\
          \downarrow& &\downarrow h|_A& &\downarrow h \\pt&\longleftarrow&
          A'&\stackrel{i'}\longrightarrow&X'\end{array}.\] The statement follows from the
          Gluing Lemma (Lemma 2.4 in \cite{WZZ}).
      \item[(ii)]
        The quotient space $S^{p}\times S^{q}/\pt\times S^{q}$ is a Thom
        space $\tau\epsilon_p$ of the trivial bundle $\epsilon_p$ over the
        sphere $S^{q}$, which is the one-point compactification
        $(S^{q}\times\RR^{p})^{\wedge}$. The lemma follows from the
        identity $(X\times Y)^{\wedge}=X^{\wedge}\wedge
        Y^{\wedge}$.
      \item[(iii)]
         The quotient space $A\times D^{n}/A\times pt$ is homotopy
         equivalent to the mapping cone of the inclusion $A\times
         pt\stackrel{i}\rightarrow A\times D^{n}$, which is homotopy
         trivial. Hence the quotient is contractible.
    \end{itemize}
  \end{proof}

  For the next lemma we need some basic facts from the theory of
  homotopy colimits. We refer the reader to \cite{WZZ} for results
  from that theory that are formulated in combinatorial language and
  for further references. The following lemma is a version of \cite[Lem. 3.3]{GT1}.

  \begin{Lemma}\label{l4}
    Let $A\stackrel{i}\rightarrow X$ and $S\stackrel{j}\rightarrow D$
    be homotopy trivial inclusions. Then the homotopy colimit space
    $Y=\hocolim\{X\times S\stackrel{i\times 1}\longleftarrow A\times
    S\stackrel{1\times j}\longrightarrow A\times D\}$ is homotopy
    equivalent to the wedge of spaces
    \[
      Y\simeq X\times S/\pt\times S\vee\Sigma(A\wedge S)\vee A\times D/A\times \pt.
    \]
  \end{Lemma}

  \begin{proof}
    By the assumption for the maps $i$ and $j$, we have
    \[Y\simeq \hocolim\{X\times S\stackrel{\pt\times 1}\longleftarrow A\times
    S\stackrel{1\times \pt}\longrightarrow A\times D\}.\] Let $A\ast
    S=A\times S\times[0,1]/\sim$ be the join of the spaces $A$ and
    $S$, where $(x,y_1,0)\sim(x,y_2,0), (x_1,y,1)\sim(x_2,y,1)$. The
    spaces $A$ and $S$ are embedded into the join by $i_2:S\rightarrow
    A\ast S$ and $j_2:A\rightarrow A\ast S$, where $i_2(y)=[(x,y,1)],
    y\in S$ and $j_2(x)=[(x,y,0)], x\in A$.

    The space $Y$ is homotopy equivalent to the colimit space

    \[Y\simeq \colim\{X\times S\stackrel{i_1}\longleftarrow S\stackrel{i_2}\longrightarrow A\ast
    S\stackrel{j_2}\longleftarrow A\stackrel{j_1}\longrightarrow
    A\times D\},\]where $i_1:S\rightarrow X\times S$ and
    $j_1:A\rightarrow A\times D$ are inclusions $i_1(y)=(\pt,y), y\in
    S$ and $j_1(x)=(x,\pt), x\in A$. Since there is a quotient map
    $p:A\ast S\rightarrow \Sigma(A\wedge S)$, which is a homotopy
    equivalence, we obtain

    \[
      Y\simeq \colim\{X\times S\stackrel{j_2}\longleftarrow
      S\stackrel{\pt}\longrightarrow\Sigma(A\wedge
      S)\stackrel{\pt}\longleftarrow A\stackrel{i_2}\longrightarrow
      A\times D\}.
    \]
  \end{proof}

  Now we are in position to provide the proof of \ref{t1}.

  \begin{proof}[Proof of \ref{t1}]
    \noindent {\sf Wedge of Spheres:}
    We prove by induction on the number $m$ of vertices of the simplicial complex $K$ that if
    $K^\dual$ is vertex decomposable then:
    $\zZ_K(D^n,S^{n-1})$ is homotopy equivalent to
    a wedge of spheres.

    For $m=1$ any simplicial complex has vertex decomposable Alexander dual and there is either no or exactly
    one critical cell different from $(-,\ldots,-)$ in $\Crit(K)$. Thus it follows immediately that $\zZ_K(D^n,S^{n-1})$ is homotopy
    equivalent to the wedge of spheres given by \eqref{eq:wedge}.

    Assume $m \geq 2$.
    We first treat the case when $K^\dual$ is a simplex over some subset $\Omega$ of the ground set $[m]$. In this case
    $K = 2^\Omega \join (2^{[m] \setminus \Omega} \setminus \{ [m] \setminus \Omega \}$.
    Thus $$\zZ_K(D^n,S^{n-1}) = (D^n)^{m-r} \times ((D^{n})^r \setminus \{(x_1,\ldots,x_m)\in (D^n)^r ~|~x_i \not\in S^{n-1} \})$$
    for $r = m - \# \Omega$. Thus $\zZ_K(D^n,S^{n-1})$ is homotopy equivalent to a $(nr-1)$-sphere.

    Next we consider the case when $K^\dual$ is vertex decomposable with shedding vertex $m$.
    The induction hypothesis asserts that for any simplicial
    complex $L$ on $k<m$ vertices with vertex-decomposable Alexander dual simplicial complex the space $\zZ_K(D^n,S^{n-1})$ is homotopy
    equivalent to a wedge of spheres given by \eqref{eq:wedge}.

    The space $\zZ_K(D^n,S^{n-1})$ can be represented as a colimit of the following diagram

    \begin{diagram}
      \label{di:colim}
                                                 &                        & \zZ_{\overline{\link_K(m)}}(D^n,S^{n-1}) &           & \\
                                                 & \ldTo^{i}              &                                          & \rdTo^{j} & \\
      \zZ_{\overline{K\setminus m}}(D^n,S^{n-1}) &                        &                                          &           & \zZ_{\star_K(m)}(D^n,S^{n-1}) \\
    \end{diagram}

    where $i$ and $j$ are inclusions and for a simplicial complex $L$ on
    $[m-1]$ we denote by $\overline{L}$ the complex $L$ considered as a simplicial
    complex over $[m]$.
    We have the following identities:

    \begin{eqnarray*}
      \zZ_{\overline{K\setminus m}}(D^n,S^{n-1}) & = & \zZ_{K\setminus m}(D^n,S^{n-1}) \times S^{n-1}, \\
      \zZ_{\overline{\link_K(m)}}(D^n,S^{n-1}) & = & \zZ_{\link_K(m)}(D^n,S^{n-1})\times S^{n-1}, \\
      \zZ_{\star_K(m)}(D^n,S^{n-1}) & = & \zZ_{\link_K(m)}(D^n,S^{n-1}) \times D^n
    \end{eqnarray*}

    \noindent Substituting in \eqref{di:colim} and using the Projection Lemma it follows that
    $\zZ_K(D^n,S^{n-1})$ is homotopy equivalent to the homotopy colimit of the following diagram

    \begin{diagram}
      \label{di:holim}
                                                     &                      & \zZ_{\link_K(m)}(D^n,S^{n-1}) \times S^{n-1} &                       & \\
                                                     & \ldTo^{i' \times \id} &                                              & \rdTo^{\id \times j'} & \\
      \zZ_{K\setminus m}(D^n,S^{n-1}) \times S^{n-1} &                      &                                              &            & \zZ_{\link_K(m)}(D^n,S^{n-1}) \times D^n\\
    \end{diagram}
    where $i=i'\times \id, j=\id\times j'$ and $i', j'$ are the inclusion maps.
    Clearly, the map $j'$ is homotopy equivalent to the constant map. By \ref{pr:trivial}
    the map $i'$ is also homotopy equivalent to the constant map. Hence
    \ref{l4} applies and shows that $\zZ_K(D^n,S^{n-1})$ is homotopy
    equivalent to a wedge of the following three spaces:
    \begin{itemize}
       \item[(a)] $\Big(\zZ_{K\setminus m}(D^n,S^{n-1}) \times S^{n-1}\Big) / \Big(\pt \times S^{n-1}\Big)$

          By induction and \ref{l3} we know that $\zZ_{K\setminus m}(D^n,S^{n-1})$ is homotopy equivalent to a
          wedge of spheres. Thus $\zZ_{K\setminus m}(D^n,S^{n-1}) \times S^{n-1}$
          is homotopy equivalent to a wedge of spheres times $S^{n-1}$ by a homotopy that is
          the identity on $\pt \times S^{n-1}$. But then there is a homotopy to a wedge of
          products of a sphere with $S^{n-1}$ quotient by $\pt \times S^{n-1}$.
          The latter is by \ref{l5} (i) homotopy equivalent to a wedge of spheres,

       \item[(b)] $\Sigma \Big( \zZ_{\link_K(m)}(D^n,S^{n-1}) \wedge S^{n-1}\Big)$

         Assume there are $\ell$ elements from $[m]$ such that $\{ \ell\}$ is not a face of $\link_K(m)$. Then by induction and \ref{l3} $\zZ_{\link_K(m)}(D^n,S^{n-1})$
         is homotopy equivalent to the product of $(S^{n-1})^\ell$ and a wedge of spheres
         given by \eqref{eq:wedge}.
         Since the smash product with $S^{n-1}$ is the $(n-1)$-fold suspension it follows that
         the space is homotopy equivalent to an $n$-fold suspension of a product of spheres and a wedge of spheres for some $n \geq 1$.
         Elementary homotopy arguments that show that it is homotopy equivalent to a wedge of spheres.

       \item[(c)] $\Big(\zZ_{\link_K(m)}(D^n,S^{n-1}) \times D^n \Big)/ \Big( \zZ_{\link_K(m)}(D^n,S^{n-1}) \times \pt \Big)$.

         This is contractible by \ref{l5} (iii).
    \end{itemize}
    Hence $\zZ_K(D^n,S^{n-1})$ is homotopy equivalent to a wedge of spheres.

    \noindent {\sf Counting the Spheres:}
    In the second part of the proof we verify the exact formula for the number of spheres.
    Since we know that $\zZ_K(D^n, S^{n-1})$ is homotopy equivalent to a wedge of spheres
    the rank of $\widetilde{\Homology}^i(\zZ_K(D^n,S^{n-1});\KK)$ counts the number
    of $i$-spheres in the wedge for any field $\KK$.

    In order to compute $\widetilde{\Homology}^i(\zZ_K;\KK)$ we now invoke
    \ref{pr:susp}. This immediately completes the proof for $n \geq 2$. For $n=1$
    the result follows from \cite[Thm. 3.1]{GPW} and the Hochster formula for
    the Betti numbers of the minimal free resolution of a monomial ideal (see
    \cite{H} for the original proof or \cite[p.226]{ER} for its statement).
  \end{proof}

  Using the notation from \ref{p2} we obtain the following corollary.

  \begin{Corollary} \label{co:appl} Let $K$ be a simplicial complex on $[m]$ such that the
     Alexander dual complex $K^\dual$ is vertex-decomposable. Then
     \begin{itemize}
       \item[(i)] If $m$ is a shedding vertex of $K^\dual$ then $$\Crit_m =
           \Crit_{m-1}^{-} \cup \Crit_{m-1}^+ \cup \Crit_{m-1\bullet}^\bullet \setminus \{ (-,\ldots,-,+) , (-,\ldots,-,\bullet) \}.$$
       \item[(ii)] \begin{eqnarray*} \widetilde{\Homology}^*(\zZ_K(D^n,S^{n-1})) & \cong & \widetilde{\Homology}^{*} (\zZ_{K \setminus m } (D^n, S^{n-1})) \oplus \\
                                                                                 &       & \widetilde{\Homology}^{*-n+1} (\zZ_{K \setminus m } (D^n, S^{n-1})) \oplus \\
                                                                                 &       & \widetilde{\Homology}^{*-n}(\zZ_{\link_K(m)}(D^n,S^{n-1})).
                   \end{eqnarray*}
       \item[(iii)] $$\zZ_K(D^n,S^{n-1}) \simeq \bigvee_{c \in \Crit(K)} S^{\dim(c)}.$$
     \end{itemize}
   \end{Corollary}

   \begin{proof}
     \begin{itemize}
       \item[(i)] Suppose that $c \in \Crit_{m-1} \cap \Crit_{m-1\bullet}$. Then by
         \ref{cor:nonface} we have that $N = \{J(c)\} \cup \supp(c)$ is a minimal nonface of both $K\setminus m$
         and $\link_K(m)$. Thus $[m]\setminus N$ is a facet of both $K^\dual \setminus m$ and $\link_{K^\dual} (m)$,
         which contradicts the condition that $m$ is a shedding vertex of $K^\dual$. We obtain that $\Crit_{m-1} \cap
         \Crit_{m-1\bullet} = \{(-,\ldots,-)\}$ and the statement follows from \ref{p2}.
      \item[(ii)] The assertion follows directly from homotopy decomposition given in the proof of \ref{t1}.
      \item[(iii)] The assertion follows from (i) and (ii) by induction on the number of vertices using the
         fact that the Morse matching reduces accordingly.
     \end{itemize}
   \end{proof}

  Finally, we provide the proofs of \ref{co:flag} and \ref{co:susp}.

  \begin{proof}[Proof of \ref{co:flag}]
    By \cite[Prop. 8]{ER} we know that if $K^{(1)}$ is chordal then $K^\dual$ is vertex-decomposable.
    Since $K$ is flag we have $\{i\} \in K$ for all $i \in [m]$. Hence the result follows from
    \ref{t1}.
  \end{proof}

  \begin{proof}[Proof of \ref{co:susp}]
    If we set $n = 2$ in \ref{pr:susp} then we obtain a homotopy equivalence
    $$\Sigma \zZ_K(D^n,S^{n-1}) \simeq \bigvee_{{{M \not\in K} \atop {M \subseteq [m]}}} \Sigma^{\#M +2} |K_M|.$$
    Now the assertion follows from \ref{t1}.
  \end{proof}

  Note that if \ref{pr:susp}, respectively \cite[Cor. 2.24]{BBCG}, holds for $n \geq 1$ then one can
  strengthen \ref{co:susp} to the assertion that the double suspension is homotopy equivalent to a wedge of spheres.

  We finish with a few examples that illustrate our constructions and results.

  \begin{Example}
    \label{ex:ex51}
    Let $K$ be a simplicial complex with vertex decomposable dual on the vertex
    set $V=[6]$, such that the set of facets of $\link_K(6)$ is
    $\{125,134,145,234,235\}$. The
    geometric realization $|\link_K(6)|$ is the M\"obius band.
    The Alexander dual complex $\link_K(6)^{\dual}=K^{\dual}\setminus 6$ is the
    $5$-gon $\{12,13,24,35,45\}$. Using \ref{p2},
    we obtain $\Crit(\link_K(6))=\{
    (-,-,-,-,-), (+,\ncell,\ncell,-,-), (+,\ncell,\ncell,+,-), (+,\ncell,-,\ncell,-), (+,\ncell,\ncell,-,+),
    (+,\ncell,\ncell,+,+),\newline (+,\ncell,-,\ncell,+), (+,-,\ncell,-,\ncell),
    (+,-,\ncell,+,\ncell), (-,+,+,\ncell,\ncell), (-,-,+,\ncell,\ncell),
    (-,+,-,\ncell,\ncell)\}.$

    \noindent The necessary condition that $6$ is the shedding vertex of $K^{\dual}$ is that there is no $1$-simplex in
    $\link_{K^{\dual}}(6)=(K\setminus 6)^{\dual}$.
    The minimal complex which satisfies this condition has to contain
    all $2$-faces on the vertex set $[5]$ and hence ${[5]\choose 3}\subset
    K\setminus 6$. This shows that $\{J(c)\}\cup \supp(c)\in K\setminus
    6$, for all $c\in \Crit(\link_K(6))$.
  \end{Example}

  \begin{figure}
    \vskip6cm
    \hskip-12cm
    \begin{picture}(0,0)%
      \includegraphics{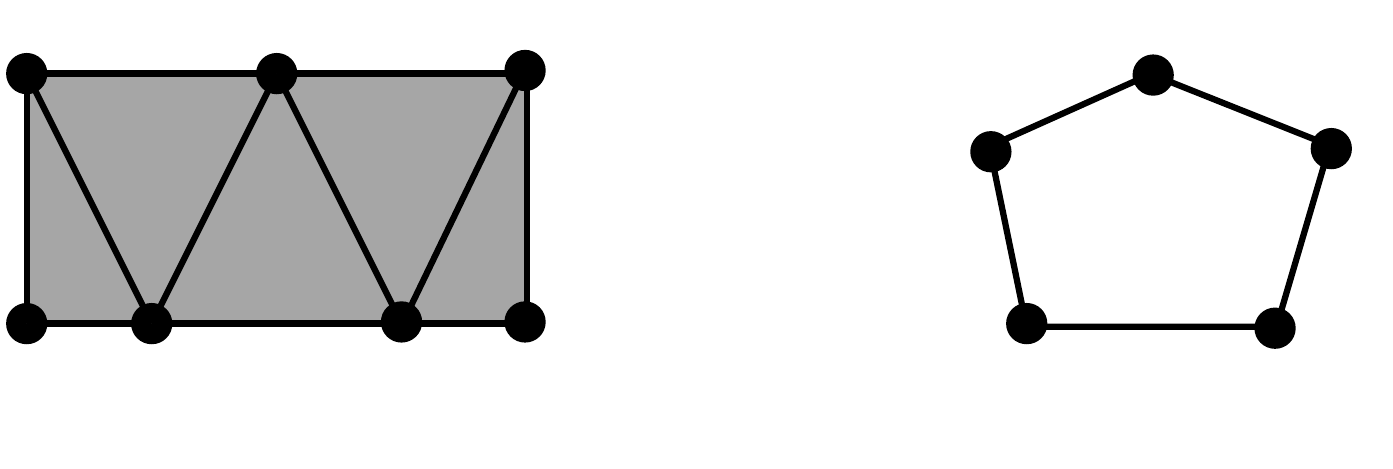}%
    \end{picture}%
    \setlength{\unitlength}{3947sp}%

    \begingroup\makeatletter\ifx\SetFigFont\undefined%
    \gdef\SetFigFont#1#2#3#4#5{%
    \reset@font\fontsize{#1}{#2pt}%
    \fontfamily{#3}\fontseries{#4}\fontshape{#5}%
    \selectfont}%
    \fi\endgroup%
    \begin{picture}(6678,2242)(1673,-2858)
      \put(2100, 1400){\makebox(0,0)[lb]{\smash{{\SetFigFont{12}{14.4}{\rmdefault}{\mddefault}{\updefault}{\color[rgb]{0,0,0}$1$}}}}}
      \put(3340, 1400){\makebox(0,0)[lb]{\smash{{\SetFigFont{12}{14.4}{\rmdefault}{\mddefault}{\updefault}{\color[rgb]{0,0,0}$2$}}}}}
      \put(3900, -100){\makebox(0,0)[lb]{\smash{{\SetFigFont{12}{14.4}{\rmdefault}{\mddefault}{\updefault}{\color[rgb]{0,0,0}$3$}}}}}
      \put(2100, -100){\makebox(0,0)[lb]{\smash{{\SetFigFont{12}{14.4}{\rmdefault}{\mddefault}{\updefault}{\color[rgb]{0,0,0}$4$}}}}}
      \put(2700, -100){\makebox(0,0)[lb]{\smash{{\SetFigFont{12}{14.4}{\rmdefault}{\mddefault}{\updefault}{\color[rgb]{0,0,0}$5$}}}}}
      \put(4800, -100){\makebox(0,0)[lb]{\smash{{\SetFigFont{12}{14.4}{\rmdefault}{\mddefault}{\updefault}{\color[rgb]{0,0,0}$1$}}}}}
      \put(4800, 1400){\makebox(0,0)[lb]{\smash{{\SetFigFont{12}{14.4}{\rmdefault}{\mddefault}{\updefault}{\color[rgb]{0,0,0}$4$}}}}}
      \put(3100, -300){\makebox(0,0)[lb]{\smash{{\SetFigFont{12}{14.4}{\rmdefault}{\mddefault}{\updefault}{\color[rgb]{0,0,0}$\link_K(6)$}}}}}

      \put(7550, 1400){\makebox(0,0)[lb]{\smash{{\SetFigFont{12}{14.4}{\rmdefault}{\mddefault}{\updefault}{\color[rgb]{0,0,0}$1$}}}}}
      \put(6700,  800){\makebox(0,0)[lb]{\smash{{\SetFigFont{12}{14.4}{\rmdefault}{\mddefault}{\updefault}{\color[rgb]{0,0,0}$2$}}}}}
      \put(8700,  800){\makebox(0,0)[lb]{\smash{{\SetFigFont{12}{14.4}{\rmdefault}{\mddefault}{\updefault}{\color[rgb]{0,0,0}$3$}}}}}
      \put(6900, -100){\makebox(0,0)[lb]{\smash{{\SetFigFont{12}{14.4}{\rmdefault}{\mddefault}{\updefault}{\color[rgb]{0,0,0}$4$}}}}}
      \put(8400, -100){\makebox(0,0)[lb]{\smash{{\SetFigFont{12}{14.4}{\rmdefault}{\mddefault}{\updefault}{\color[rgb]{0,0,0}$5$}}}}}
      \put(7350, -300){\makebox(0,0)[lb]{\smash{{\SetFigFont{12}{14.4}{\rmdefault}{\mddefault}{\updefault}{\color[rgb]{0,0,0}$\link_K(6)^\dual$}}}}}
    \end{picture}%
    \vskip-5cm
    \caption{Complexes from \ref{ex:ex51}}
  \end{figure}

  \begin{Example}
    Let $K_1$ and $K_2$ be simplicial complexes, which are proper
    subcomplexes of two simplices on disjoint ground sets.
    Since the $n$-sphere moment-angle complex corresponding to the
    join $K_1\join K_2$ is the product $\zZ_{K_1\join
    K_2}(X,A)=\zZ_{K_1}(X,A)\times\zZ_{K_2}(X,A)$, it follows from
    \ref{t1} that the dual complex of the join $(K_1\join K_2)^\dual$
    is not vertex decomposable in general.
  \end{Example}

  \begin{Example}
    Let $K$ be the $(k-1)$-skeleton of the
    simplex on $m$ vertices. Then the dual complex $K^\dual$ is
    the $(m-k-1)$-skeleton of the simplex on $m$ vertices.
    It is easily seen that any skeleton of a simplex is vertex decomposable.
    Hence \ref{t1} applies. For any $M \subseteq [m]$ the complex
    $K_M$ is the $(k-1)$-skeleton of the full simplex with vertex set $M$.
    In particular, $K_M$ is the full simplex if $\# M \leq k$.
    If $\#M >k$ then the homology of $K_M$ is of rank ${\# M -i \choose
    k-1}$ in dimension $k-1$.
    If we apply \ref{t1} to our situation then we get
    the following homotopy equivalence which is also an immediate
    consequence of \cite{GT1}, \cite{GT2}
    \[\zZ_K(D^{n}, S^{n-1})\simeq \bigvee_{j=1}^{m-k} {m\choose k+j}{k+j-j \choose k}S^{nk+j(n-1)}.\]
  \end{Example}

\section{Acknowledgements} We thank Kouyemon Iriye and Daisuke Kishimoto for pointing us to a mistake in the
  formulation of \ref{pr:trivial} and a sloppy argument in the proof of \ref{l5} (i).
  We also thank Jelena Grbi\'c for mentioning to us that \ref{t2} also follows easily
  from the fact that $\zZ_K(\bullet,\bullet)$
  is a homotopy functor in both arguments independently.






\bibliographystyle{model1a-num-names}
\bibliography{<your-bib-database>}

\end{document}